\documentclass{amsart}
\usepackage[latin9]{inputenc}
\usepackage{color}
\usepackage{dsfont}
\usepackage{amsthm}
\usepackage{amssymb}
\usepackage[unicode=true,
 bookmarks=false,
 breaklinks=false,pdfborder={0 0 1},backref=section,colorlinks=true]
 {hyperref}
\hypersetup{
 linkcolor=blue,filecolor=magenta,urlcolor=cyan}

\makeatletter
\numberwithin{equation}{section}
\numberwithin{figure}{section}
\theoremstyle{plain}
\newtheorem{thm}{\protect\theoremname}
\theoremstyle{definition}
\newtheorem{defn}[thm]{\protect\definitionname}
\theoremstyle{plain}
\newtheorem{lem}[thm]{\protect\lemmaname}
\theoremstyle{remark}
\newtheorem{rem}[thm]{\protect\remarkname}
\theoremstyle{plain}
\newtheorem{prop}[thm]{\protect\propositionname}
\theoremstyle{definition}
\newtheorem{example}[thm]{\protect\examplename}
\theoremstyle{plain}
\newtheorem{cor}[thm]{\protect\corollaryname}

\usepackage{graphics}

\usepackage{pdfpages}

\theoremstyle{definition}

\theoremstyle{corollary}

\theoremstyle{note}

\theoremstyle{remark}

\numberwithin{equation}{section}

\newcommand{\bs}{\backslash}

\newcommand{\imply}{\Rightarrow}
\newcommand{\script}[1]{\mathcal{#1}}

\makeatletter
\@namedef{subjclassname@2010}{%
  \textup{2010} Mathematics Subject Classification}
\makeatother

\makeatother

\providecommand{\corollaryname}{Corollary}
\providecommand{\definitionname}{Definition}
\providecommand{\examplename}{Example}
\providecommand{\lemmaname}{Lemma}
\providecommand{\propositionname}{Proposition}
\providecommand{\remarkname}{Remark}
\providecommand{\theoremname}{Theorem}

\begin{document}
\title{Construction of nearly pseudocompactifications}
\author{Biswajit Mitra }
\address{Department of Mathematics. The University of Burdwan, Burdwan 713104,
West Bengal, India }
\email{bmitra@math.buruniv.ac.in}
\author{Sanjib Das }
\address{Department of Mathematics. The University of Burdwan, Burdwan 713104,
West Bengal, India}
\email{ruin.sanjibdas893@gmail.com}
\thanks{The second author's research grant is supported by CSIR, Human Resource
Development group, New Delhi-110012, India }
\begin{abstract}
A space is nearly pseudocompact if and only if $\upsilon X\bs X$
is dense in $\beta X\bs X$. If we denote $K=cl_{\beta X}(\upsilon X\bs X)$,
then $\delta X=X\cup(\beta X\bs K)$ is referred by Henriksen and
Rayburn \cite{hr80} as nearly pseudocompact extension of $X$. Henriksen
and Rayburn studied the nearly pseudocompact extension using different
properties of $\beta X$. In this paper our main motivation is to
construct nearly pseudocompact extension of $X$ independently and
not using any kind of extension property of $\beta X$. An alternative
construction of $\beta X$ is made by taking the family of all $z$-ultrafilters
on $X$ and then topologized in a suitable way. In this paper we also
adopted the similar idea of constructing the $\delta X$ from the
scratch, taking the collection of all $z$-ultrafilters on $X$ of
some kind, called $hz$-ultrafilters, together with fixed $z$-ultrafilter
and then be topologized in the similar way what we do in the construction
of $\beta X.$ We have further shown that the extension $\delta X$
is unique with respect to certain properties..
\end{abstract}

\maketitle

\section{Introduction}

Through out the paper, by a space we shall always mean Tychonoff.
For a space $X$, let $C(X)$ and $C^{*}(X)$ be the rings of real-valued
continuous and bounded continuous functions respectively. The Stone-$\check{C}$ech
compactification of $X$, usually denoted as $\beta X,$ is the unique
compactification of $X$ with respect to the extension property in
the sense that every compact-valued continuous function can be uniquely
extended to $\beta X$. Similarly the Hewitt realcompactification
$\upsilon X$ of $X$ is the unique realcompactification of $X$ in
the sense that every realcompact-valued continuous function can be
continuously extended to $\upsilon X$. In the year 1980, Henriksen
and Rayburn \cite{hr80} introduced nearly pseudocompact spaces. A
space $X$ is nearly pseudocompact if and only if $\upsilon X\backslash X$
is dense in $\beta X\bs X$. If we denote $K=cl_{\beta X}(\upsilon X\bs X)$.
Then $\delta X=X\cup(\beta X\bs K)$ is referred by Henriksen and
Rayburn \cite{hr80} as nearly pseudocompact extension of $X$. Henriksen
and Rayburn studied the nearly pseudocompact extension using different
properties of $\beta X$ in \cite{hr87}. In this paper our main motivation
is to construct nearly pseudocompact extension of $X$ independently
and not using any kind of extension property of $\beta X$. In chapter
6 of \cite{gj60}, the construction of $\beta X$ was made by taking
the family of all $z$-ultrafilters on $X$ and then topologized in
a suitable way. In this paper we also adopted the similar idea of
constructing the $\delta X$ by taking the collection of all $hz$-ultrafilters
together with fixed z-ultrafilters on $X$ and then topologized in
the similar way what we do in the construction of $\beta X.$ We have
further shown that the extension $\delta X$ is unique with respect
to certain properties.

In section 2, we discussed about few preliminary topics. In section
3, we initially started with an equivalent definitions of hard sets
and nearly pseudocomapact spaces where $\beta X$ is not involved.
Then we have provided alternative proofs of those theorems where properties
of $\beta X$ are involved in the original proof. We had to do this
intending to avoid any kind of cyclic arguments. Finally we constructed
$\delta X$ as collection of all $hz$-ultrafilters and fixed $z$-ultrafilters
and topologized it and proved that $\delta X$ is indeed a nearly
pseudocompactification of a space $X.$ In section 4, we tried to
explore extension properties of certain kinds of maps such as $h2pc$
map, hard map and have shown that $\delta X$ is unique with respect
to some properties.

\section{Preliminaries}

In this paper, we used the most of preliminary concepts, notations
and terminologies from the classic monograph of L.Gillman and M.Jerison,
Rings of Continuous Functions \cite{gj60}. However for ready references,
we recall few notations, frequently used over here. For any $f\in C(X)$
or $C^{*}(X)$, $Z(f)=\{x\in X:f(x)=0\}$, called zero set of $f$.
Complement of zero set is called cozero set or cozero part of $f$,
denoted as $cozf$. For any $f\in C(X)$ or $C^{*}(X)$, $cl_{X}(X\bs Z(f))$
is called the support of $f$. A space is realcompact if and only
if every $z$-ultrafilter with countable intersection property is
fixed. In the year 1976, Rayburn \cite{r76} introduced hard set.
\begin{defn}
A subspace $H$ of $X$ is called hard in $X$ if $H$ is closed in
$X\cup cl_{\beta X}(\upsilon X\bs X).$
\end{defn}

It immediately follows that every hard set is closed in $X$, but
the converse is obviously not true. In the same paper {[}Lemma 1.3,
\cite{r76}{]}, Rayburn proved the following characterization of hard
subsets of $X$.
\begin{thm}
A closed set $H$ is hard in $X$ if and only if there exists a compact
set $K$ of $X$ such that for any open neighbourhood $G$ of $K$,
$H\bs G$ is completely separated from the complement of some realcompact
set.
\end{thm}

In particular, it follows that if a closed set is completely separated
from complement of some realcompact space, then the set is hard in
$X$. However the converse may not be true. The converse is also true
if $X$ is locally realcompact, that is, for each point $x$, there
is a realcompact neighbourhood of $x$ {[}Corollary 1.5, \cite{r76}{]}.

In the year 1980, Henriksen and Rayburn \cite{hr80} introduced nearly
pseudocompact space. A space $X$ is nearly pseudocompact if $\upsilon X\bs X$
is dense in $\beta X\bs X$. It is immediate that every pseudocompact
space is nearly pseudocompact. The converse is not true in general.
In the same paper they have shown that every anti-locally realcompact
space (i.e. the space with no points having realcompact neighbourhood)
is nearly pseudocompact {[} Corollary 3.5, \cite{hr80}{]} and produced
a series of examples of anti-locally realcompact space which are not
pseudocompact {[}Proposition 3.6, \cite{hr80}{]}. However in {[}\cite{mc20},
Example 4.14{]}, the author gave an example of nearly pseudocompact
but not pseudocompact, not even anti-locally realcompact.

Henriksen and Rayburn in their paper {[}\cite{hr80} Theorem 3.2{]}
furnished the following characterizations of nearly pseudocompact
spaces.
\begin{thm}
The followings are equivalent.
\end{thm}

\begin{enumerate}
\item $X$ is nearly pseudocompact
\item Every hard set is compact.
\item Every regular hard set is compact.
\item Each decreasing sequence of non-empty regular hard sets has non-empty
intersection.
\item $X$ can be expressed as $X_{1}\cup X_{2},$ where $X_{1}$ is a regular
closed almost locally compact pseudocompact subset and $X_{2}$ is
regular closed anti-locally realcompact and $int_{X}(X_{1}\cap X_{2})=\emptyset$,
where a subset of a space $T$ is almost locally compact if it is
the $T$-closure of the set of all points in $T$ having compact neighbourhoods.
\end{enumerate}
Mitra and Acharyya in \cite{ma05} worked further on nearly pseudocompact
spaces using the $C_{H}(X)$ and $H_{\infty}(X)$. $C_{H}(X)$ is
the subring of $C(X)$ consisting of all those members of $C(X)$
which have hard support. $H_{\infty}(X)$ is the family of all real-valued
continuous functions so that the set $\{x\in X:|f(X)|\geq\frac{1}{n}\}$
is hard in $X$ for all $n\in\mathbb{N}$. $H_{\infty}(X)$ is also
a subring of $C(X)$ and $C_{H}(X)\subseteq H_{\infty}(X)$. In \cite{mc20},
Mitra and Chowdhury showed that $H_{\infty}(X)$ is precisely $C_{RC}(X),$
the subring of $C(X)$ which is ideal of $C(X)$ too, consisting of
all those continuous maps whose cozero part are realcompact. A result
related to $C$-embeddedness has been used here. As every $C$-embedded
subset is $C^{*}$-embedded, any zero set disjoint with a $C$-embedded
subset can be completely separated {[}Chapter 1, \cite{gj60}{]}.

A subset of $X$ is called $z$-embedded in a space $Y$ if every
zero set in $X$ is the trace of some zero set in $Y$ on $X.$ It
is indeed a generalized notion of both $C$-embedding and $C^{*}$-embedding.
Every cozero set in $X$ is $z$-embedded in $X.$ Countable union
of $z$-embedded real compact subsets is again realcompact. All the
literature of $z$-embeddedness are available in the book, Hewitt-Nachbin
Spaces by M.Weir {[}Page 108- 120, \cite{mw75}{]}.

We also recall few facts about $z$-filter, prime $z$-filter and
$z$-ultrafilter. A $z$-filter $\script{F}$ is called fixed if $\cap\script{F}\neq\emptyset$.
Otherwise it is called free $z$-filter. Lemma 4.10 of \cite{gj60}
tells that a zero set $Z$ is compact if and only if $Z$ is not a
member of any free $z$-filter. $z$-ultrafilters are maximal $z$-filters
and prime $z$-filters are those $z$-filters where if union of two
zero sets is a member of the $z$-filter then any one of them must
be a member of that $z$-filter. A prime $z$-filter $\script{F}$
converges to a point $p$ if and only if it clusters at the point
$p$ if and only if $\cap\script{F}=\{p\}$. Every $z$-ultrafilter
is a prime $z$-filter. However the converse may not be true. Fixed
$z$-ultrafilters on $X$ are precisely of the form $\{A_{x}:x\in X\}$,
where $A_{x}\mathrel{:=}\{Z\in Z[X]:x\in Z\}$. From theorem 5.7 of
\cite{gj60} it follows that a $z$-ultrafilter is real if and only
if for any $f\in C(X)$ there exists a natural number $n$ such that
$Z_{n}(f)\notin\mathcal{U}$ where $Z_{n}(f)=\{x\in X:|f(x)|\geq n\}$.
Let $\tau:X\to Y$ be a continuous map. If $\mathcal{F}$ be a $z$-filter
on $X$ then $\tau^{\#}(\mathcal{F})=\{Z\in Z(X):\tau^{-1}(Z)\in\mathcal{F}\}$
is a $z$-filter on $Y$. If $\mathcal{F}$ is prime then $\tau^{\#}(\mathcal{F})$
is prime. However analogous version of $z$-ultrafilter does not hold
true. If $S\subset X$ and $\mathcal{F}$ be a $z$-ultrafilter on
$X$ then $i^{\#}(\mathcal{F})\cap S=\mathcal{F}$, where $i:S\to X$
is the inclusion map and $i^{\#}(\mathcal{F})\cap S=\{Z\cap S:Z\in i^{\#}(\mathcal{F})\}$.

\section{Construction of nearly pseudocompactifications}

In the year 1987, Henriksen and Rayburn in their paper \cite{hr87}
discussed about nearly pseudocompact extension. They call $X\cup(\beta X\backslash K)$
as nearly pseudocompact extension of $X$, where $K=cl_{\beta X}(\upsilon X\bs X)$,
denoted by $\delta X$ and studied different properties of this extension.
For instance a set $H$ is hard in $X$ if and only if there exists
a compact set $K$ in $\delta X$ so that $K\cap X=H$. But all these
proofs were done by crucially using the various properties, specially
the extension property of $\beta X$. In this section we shall try
to explore a direct construction of nearly pseudocompactification
$\delta X$. Likewise the construction of $\beta X$ in chapter 6
of \cite{gj60}, we shall almost follow the same lines of arguments.
In this paper we may only use $\beta X$ as a collection of all $z$-ultrafilters
on $X$ which is Tychonoff and nothing more. In consequence, we hereby
recall the construction of $\beta X$. If $Z[X]$ denotes the family
of all zero sets in $X$, then $\{\overline{Z}:Z\in Z[X]\}$ forms
a base for closed sets for some topology on $\beta X$ which is compact
and Hausdorff and hence Tychonoff in which $X$ is densely embedded,
where $\overline{Z}=\{p\in\beta X:Z\in p\}$. In order to avoid junk
of notations, we shall denote fixed $z$-ultrafilters by the corresponding
points of $X.$ So when $x\in X,$ we shall treat it as per requirement
either as an element of $X$ or a fixed $z$-ultrafilter $A_{x}\mathrel{:=}\{Z\in Z[X]:x\in Z\}$
that converges to the point $x.$ In this frame work, $x\in Z\iff Z\in x$
is no more self-contradictory statement and we denote this statement
by ($*$) for future references. Likewise we shall construct nearly
pseudocompactification of $X$ as a collection of certain type of
$z$-ultrafilters. But prior to that we redefine some relevant notions
and alternate proofs of few results to be used here without involving
$\beta X$ as a Stone-$\check{C}$ech compactification of $X$.

Let $RX$ denotes the family of all free real $z$-ultrafilters on
$X$ and $HRX$ denotes the family of all hyper-real $z$-ultrafilters.
Then $\beta X\backslash X$ is the union of $RX$ and $HRX.$

We already know that a space $X$ is pseudocompact if every real-valued
continuous function on $X$ is bounded. However we hired the following
definition of realcompact space from Mandelker {[}\cite{m71},Theorem
5.1{]}, who used this definition to study some likewise properties
of realcompactness in non-Tychonoff set up.
\begin{defn}
\label{def5}A space $X$ is called realcompact if every stable family
of closed sets with finite intersection property is fixed.
\end{defn}

It immediately follows {[}Theorem 5.2,\cite{m71}{]} that every compact
set is realcompact, every closed subset of a realcompact space is
realcompact and a space is compact if and only if it is pseudocompact
and realcompact. The following theorem is given in chapter 5 and chapter
8 of \cite{gj60}. But since our definition of realcompact is different
from the definition given in \cite{gj60}, we give here an alternate
proof of the following theorem.
\begin{thm}
\label{thm6}The followings are equivalent.
\end{thm}

\begin{enumerate}
\item A space is realcompact.
\item Any prime $z$-filter with countable intersection property is fixed.
\item Any $z$-ultrafilter with countable intersection property is fixed.
\end{enumerate}
\begin{proof}
$(1)\Rightarrow(2)$ : Let $\mathcal{F}$ be a prime $z$-filter on
$X$ with countable intersection property. Let for each $n$, $Z_{n}(f)=\{x\in X:|f(x)|\geq n\}$
and $Z_{n}^{'}(f)=\{x\in X:|f(x)|\leq n\}$. Then for each $f\in C(X)$
there exists $n$ such that $Z_{n}(f)\notin\mathcal{F}$. Then $Z_{n}^{'}(f)\in\mathcal{\mathcal{F}}$
as $\mathcal{F}$ is prime. So $\mathcal{F}$ turns out to be stable
family of closed subsets $X$. Hence it is fixed.

$(2)\Rightarrow(3):$ Trivial.

$(3)\Rightarrow(1):$ Let $\mathcal{S}$ be a stable family of closed
subsets of $X$. $\mathcal{F}=\{Z\in Z[X]:Z\supseteq F,\mbox{for some \ensuremath{F\in\mathcal{S}\}.}}$
Then $\mathcal{F}$ is a family of zero sets having finite intersection
property. Hence it can be extended to a $z$-ultrafilter $\mathcal{U}.$
Then for each $f\in C(X),$ there exists a $D\in\mathcal{S},$ such
that $f$ is bounded on $D$. So there exists an $n$ such that $Z_{n}^{'}(f)\supseteq D$
and hence $Z_{n}^{'}(f)\in\mathcal{F}$. So $Z_{n+1}(f)\notin\mathcal{U}.$
So $\mathcal{U}$ is a real $z-$ultrafilter. Hence, $\bigcap\mathcal{U}\neq\emptyset.$
As $X$ is Tychonoff, the family of zero sets form a base for closed
sets and therefore $\bigcap\mathcal{U}\subseteq\cap\mathcal{F}=\cap\mathcal{S}.$
So $\mathcal{\cap S}\neq\emptyset$. Hence $X$ is realcompact.
\end{proof}
The following lemma is very useful in this paper and can be easily
derived from the proof of Lemma 8.12 of \cite{gj60}.
\begin{lem}
\label{lem7}Any prime $z$- filter with countable intersection property
is contained in unique real-$z$-ultrafilter.
\end{lem}

The following theorem was proved in \cite{gj60}, corollary 8.14 involving
intricately the extension property of $\beta X$ and $\upsilon X.$
Here we have proved this results without using any kind of extension
property of $\beta X$.
\begin{thm}
\label{thm8}Every cozero subset of a realcompact space is realcompact.
\end{thm}

\begin{proof}
Let $S$ be a cozero subset of $X.$ Suppose $X\bs S=Z(f)$ for some
$f\in C(X).$ Let $\mathcal{F}$ be a real-$z$-ultrafilter on $S.$
Then $i^{\#}(\mathcal{F})$ is a prime-$z$-filter on $X$ and $i^{\#}(\mathcal{F})\cap S=\mathcal{F}$
as $S$ is $z$-embedded in $X.$ It is contained in the unique real
$z$-ultrafilter $\mathcal{U}$ on $X.$ If $Z(f)\in\mathcal{U},$
then $\{x\in X:|f(x)|\leq\frac{1}{n}\}$ is a member of $\mathcal{U},\forall n$
and $\{x\in X:|f(x)|\geq\frac{1}{n}\}\notin\mathcal{U},$ for all
$n$. As $i^{\#}(\mathcal{F})$ is a prime $z$-filter, $\{x\in X:|f(x)|\leq\frac{1}{n}\}\in i^{\#}(\mathcal{F})$,
for all $n$. Hence $\{x\in X:|f(x)|\leq\frac{1}{n}\}\cap S\in\mathcal{F}.$
As $\mathcal{F}$ is real $z$-ultrafilter on $S$, $\bigcap_{n}\{x\in X:|f(x)|\leq\frac{1}{n}\}\cap S\neq\emptyset$.
So there exists a point $z\in S$ such that $|f(z)|\leq\frac{1}{n},$
for all $n$. Hence $f(z)=0,$ which is a contradiction. So $Z(f)\notin\mathcal{U}.$
Then $(\bigcap\mathcal{U})\cap Z(f)=\emptyset,$ otherwise $Z(f)$
would intersect every member of $\mathcal{U}.$ This would imply that
$Z(f)\in\mathcal{U}$, a contradiction. As $\mathcal{U}$ is real
and $X$ is realcompact, $\cap\mathcal{U}\neq\emptyset.$ Now $\mathcal{\cap U\subseteq\cap}i^{\#}(\mathcal{F})$
$\Rightarrow$ $\cap\mathcal{U}=(\cap\mathcal{U})\cap S\subseteq(\cap i^{\#}(\mathcal{F}))\cap S=\cap(i^{\#}(\mathcal{F})\cap S)=\cap\mathcal{F}.$
This follows that $\mathcal{F}$ is fixed. Hence $S$ is realcompact.
\end{proof}
Now we shall give definition of a hard set. As the original definition
of it involves $\beta X,$ we therefore adopted the intrinsic characterization
of hard set given by Rayburn\cite{r76}, metioned in Theorem 2 of
this paper as a definition of hard set.
\begin{defn}
\label{def11}A subset $H$ of a space $X$ is called hard in $X$
if it is closed in $X$ and there exists a compact set $K$ such that
for any open neighbourhood $V$ of $K$ there exists a realcompact
subset $P$ such that $H\bs V$ is completely separated from the complement
of $P$ . Henceforth we shall refer $K$ as the compact set satisfying
the property of hardness.
\end{defn}

\begin{rem}
\label{rem12}It immediately follows that a closed subset of $X$
which is completely separated from a complement of realcompact subset
containing the closed set, is hard in $X$ (Simply take $K=\emptyset$).
Hence any zero set contained in a realcompact cozero set is hard.
Moreover every closed subset of hard set is hard as if $L$ is a closed
subset of a hard set $H$ and $K$ is the compact set satisfying the
proprty of hardness of $H$ , then same $K$ will satisfy the property
of hardness of $L$.

The following theorem is a little modification of the definintion
of hard set but equally important in our paper.
\end{rem}

\begin{thm}
\label{thm13}A closed set $H$ is hard in $X$ if and only if there
exists compact subset $K$ of $H$ such that for any open neighbourhood
$G$ of $K$, $H\bs G$ is completely separated from the complement
of some realcompact set.
\end{thm}

\begin{proof}
Suppose $H$ be a closed set in $X$ which is hard in $X$.

We have to show that there exists a compact set $K$ in $H$ which
satisfies the property of hardness. Since $H$ is hard, there exists
a compact subset $L$ of $X$ which satisfies the property of hardness.
$X$ being Tychonoff, $L$ is closed and so is $H\cap L$. Again $H\cap L\subset L$
which implies $H\cap L$ is compact subset of $H$, we call it as
$K$. Now, our claim is, $K$ satisfies the property of hardness.
Let $V$ be an open neighbourhood of $K$. Then $L\bs V$ is compact
subset of $X$ and $H$ is closed in $X$ with $H\cap(L\bs V)=\phi$.
Then $H$ and $L\bs V$ can be strongly separated. So, there exists
an open set $U$ in $X$ containing $L\bs V$ and $H\cap U=\phi$.
Now $V\cup U$ is an open set in $X$ such that $L\subset V\cup U$
i.e. $V\cup U$ is an open neighbourhood of $L$. Therefore by the
property of hardness of $L$, $H\bs(V\cup U)$ is completely separated
from the complement of a realcompact subset of $X$. But $H\bs(V\cup U)=(H\bs V)\cap(H\bs U)=H\bs V,$
as $H\cap U=\phi.$ So $H\bs V$ is completely separated from the
complement of a realcompact subset of $X$. Hence we have a compact
subset $K$ of $H$ which satisfies the property of hardness.

Conversely, let $H$ be a closed subset of $X$ and there exists a
compact subset $K$ of $H$ such that for any open neighbourhood $G$
of $K$, $H\bs G$ is completely separated from the complement of
some realcompact subset of $X$.

Since $K$ is compact in $H$, it is compact in $X$ also. Hence we
have a compact set $K$ in $X$ such that the given condition holds.
Therefore $H$ is hard in $X$.
\end{proof}
The following trivial but important observations are indeed direct
consequences of the Theorem \ref{thm13}.
\begin{thm}
Let $H$ be a closed subset of $X$. Then the followings are equivalent.
\end{thm}

\begin{enumerate}
\item $H$ is hard in $X$.
\item There exists a compact subset $K$ of $H$ such that for any open
neighbourhood $V$ of $K$, there exist a zero set $Z_{1}$ and a
realcompact cozero set $X\bs Z_{2}$ such that $H\bs V\subset Z_{1}\subset X\bs Z_{2}$.
\item There exists a compact subset $K$ of $H$ such that for any open
neighbourhood $V$ of $K$, there exist a regular hard zero set $Z_{1}$
and a realcompact cozero set $Z_{2}$ with $H\bs V\subset Z_{1}\subset X\bs Z_{2}$.
\end{enumerate}
\begin{proof}
$(1)\Rightarrow(2):$ Let $H$ is hard in $X$. So there exists a
compact subset $K$ of $H$ such that for any open neighbourhood $V$
of $K$, $H\bs V$ is completely separated from the complement of
a realcompact subspace $P$ of $X$. i.e. $H\bs V$ and $X\bs P$
are completely separated. Then there exist two disjoint zero sets
$Z_{1}$ and $Z_{2}$ such that $H\bs V\subset Z_{1}$ and $X\bs P\subset Z_{2}$
with $Z_{1}\cap Z_{2}=\phi$. So $Z_{1}\subset X\bs Z_{2}$ and $X\bs P\subset Z_{2}\Rightarrow X\bs Z_{2}\subset P$.
As $P$ is realcompact and $X\bs Z_{2}$ is a cozero subset of $P$,
this implies $X\bs Z_{2}$ is a realcompact cozero set in $X$ and
$H\bs V\subset Z_{1}\subset X\bs Z_{2}$.

$(2)\Rightarrow(1):$ Let there exists a compact set $K$ of $H$
such that for any open neighbourhood $V$ of $K$, there exist a zero
set $Z_{1}$ and a realcompact cozero set $X\bs Z_{2}$ such that
$H\bs V\subset Z_{1}\subset X\bs Z_{2}$. Therefore $Z_{1}\cap Z_{2}=\phi$,
then $H\bs V$and $Z_{2}$ are completely separated where $Z_{2}$
is the complement of a realcompact subset $X\bs Z_{2}$. Hence $H$
is hard in $X$.

$(1)\Rightarrow(3):$ Let $H$ is hard in $X$. Then as $(1)\Rightarrow(2)$
and $(2)\Rightarrow(1)$, so there exists a compact set $K$ in $H$
such that for any open neighbourhood $V$ of $K$, there exist a zero
set $Z'_{1}$ and a real compact cozero set $X\bs Z_{2}$ such that
$H\bs V\subset Z'_{1}\subset X\bs Z_{2}$. i.e. $Z'_{1}\cap Z_{2}=\phi$,
so $Z'_{1}$ and $Z_{2}$ are completely separated. Then there exists
$g\in C(X)$ such that $g(Z'_{1})=\{0\}$ and $g(Z_{2})=\{1\}$. Let
$Z_{1}=\{x\in X:g(x)\le\frac{1}{2}\}$ and $Z'_{2}=\{x\in X:g(x)\ge\frac{1}{2}\}$.
Therefore $Z'_{1}\subset Z_{1}$, $Z_{2}\subset Z'_{2}$ and $Z_{1}\cap Z_{2}=\phi$.
Also $cl_{X}int_{X}Z_{1}=Z_{1}\Rightarrow Z_{1}$ is a regular closed
zero set and $Z_{1}\subset X\bs Z_{2}$ where $X\bs Z_{2}$ is a realcompact
cozero set in $X$. Hence $Z_{1}$ is a regular hard zero set and
$X\bs Z_{2}$ is a realcompact cozero set such that $H\bs V\subset Z_{1}\subset X\bs Z_{2}$.

$(3)\Rightarrow(1):$ Let the given condition holds. i.e $H\bs V\subset Z_{1}\subset X\bs Z_{2}$.
Then $Z_{1}\cap Z_{2}=\phi$. So $H\bs V$ and $Z_{2}$ are completely
separated where $Z_{2}$ is the complement of a realcompact subset
of $X$. Hence $H$ is hard in $X$.
\end{proof}
Likewise the definition of hard set, we now give definition of nearly
pseudocompact space. Again its original definition involves $\beta X$.
We here adopted the intrinsic characterization of nearly pseudocompact
space given by Henriksen and Rayburn in their paper \cite{hr80},
mentioned in Theorem 3 (5) in this paper.
\begin{defn}
\label{def14}A space $X$ is nearly pseudocompact if $X$ can be
expressed as $X_{1}\cup X_{2},$ where $X_{1}$ is a regular closed
almost locally compact pseudocompact subset and $X_{2}$ is regular
closed anti-locally realcompact and $int_{X}(X_{1}\cap X_{2})=\emptyset.$
\end{defn}

The following facts are immediate from the definition of the nearly
pseudocompact space. If $X$ is almost locally compact nearly pseudocompact
space then $X=X_{1}$. Hence $X$ is pseudocompact. On the other side
if $X$ is anti-locally realcompact space then we may take $X_{1}=\phi$,
which follows from the definition that $X$ is nearly pseudocompact.
We have the following theorem already done by Henriksen and Rayburn
in \cite{hr80}. But they have proved the theorem involving $\beta X$.
Here we have given an alternate proof without using $\beta X$.
\begin{thm}
\label{thm15}The followings are equivalent.
\end{thm}

\begin{enumerate}
\item A space $X$ is nearly pseudocompact.
\item Every hard set is compact.
\item Every regular hard set is compact.
\end{enumerate}
\begin{proof}
$(1)\Rightarrow(2):$ Suppose $X$ is nearly pseudocompact. $X=$$X_{1}\cup X_{2},$
where $X_{1}$ is a regular closed almost locally compact pseudocompact
subset and $X_{2}$ is regular closed anti-locally realcompact and
$int_{X}(X_{1}\cap X_{2})=\emptyset.$ Let $H$ be a hard set. There
exists a compact subset $K\subseteq H$ satisfying the property of
hardness. Let $\mathcal{U}=\{U_{\alpha}:\alpha\in\lambda\}$ be an
open cover of $H.$ There exist finitely many indices $\alpha_{1},\alpha_{2},...,\alpha_{n}$
such that $K\subseteq\cup_{i=1}^{n}U_{\alpha_{i}}=V(say).$ Then there
exists a regular zero set $Z$ such that $H\bs V\subseteq int_{X}Z\subseteq Z$
and $Z$ is realcompact. That means $int_{X}Z\cap X_{2}=\emptyset,$
otherwise $X_{2}$ would contain a point having realcompact neighbourhood,
a contradiction. So $int_{X}Z\subseteq X_{1}.$ Then $cl_{X}int_{X}Z$
being a regular closed subset of the pseudocompact space $X_{1}$
and $Z,$ is both pseudocompact and realcompact. Hence it is compact.
Now $H\bs V$ being a closed subset of $cl_{X}int_{X}Z$, is also
compact and to cover it we need another finite sub-collection of $\mathcal{U}$
and hence as a conclusion there exists atmost finitely many members
of $\mathcal{U},$ to cover $H.$ Hence $H$ is compact.

$(2)\Rightarrow(3):$Trivial

$(3)\Rightarrow(1)$ : Suppose $X$ is not nearly pseudocompact. Then
$X\neq X_{2},$ otherwise $X$ would be anti-locally realcompact and
hence nearly pseudocompact. Then $X_{1}\neq\emptyset$ and also $X_{1}$
is not pseudocompact and hence $int_{X}X_{1}$ is not relatively pseudocompact
as the closure of open relatively pseudocompact subset is pseudocompact.
There exists $f\in C(X)$ such that $f\geq1$ and is unbounded on
$int_{X}X_{1}$ and hence contains a copy $N$ of $\mathbb{N}$ along
which $f$ tends to infinity. Now for each $n\in N,$ there exists
a cozero set $R_{n}$, such that $cl_{X}R_{n}$ is compact and is
contained in $X_{1}$. So $N\subseteq\cup_{n}R_{n}\subseteq\cup_{n}cl_{X}R_{n},$
which is a $\sigma$-compact subset of $X$ and hence realcompact.
As countable union of cozero set is again a cozero set, $\cup_{n}R_{n}=P(say)$,
a cozero set contained in the realcompact set $\cup_{n}cl_{X}R_{n}$.
Hence $P$ is a realcompact cozero set containing $N.$ As $N$ is
$C$-embedded in $X,$ $N$ and $X\bs P$ can be completely separated.
So there exists a regular zero set $W$ such that $N\subset W\subset P$.
Hence $W$ is a regular hard zero set in $X$. As $N$ is $C$-embedded
in $X$, $N$ is closed in $X$ and hence in $W$. This follows that
$W$ is not compact, a contradiction. So $X$ is nearly pseudocompact.
\end{proof}
The following theorem is already proved by Mitra and Acharyya {[}Theorem
2.5, \cite{ma05}{]} but again using $\beta X$. Here we give an alternate
proof not involving $\beta X$.
\begin{thm}
\label{th16}$X$ is nearly pseudocompact if and only if $C_{H}(X)\subseteq C^{*}(X),$where
$C_{H}(X)=\{f\in C(X):\mbox{support of }f$ $\mbox{is hard in }X\}.$
\end{thm}

\begin{proof}
If $X$ is nearly pseudocompact, every hard set is compact. So $C_{H}(X)$=
$C_{K}(X)\subseteq C^{*}(X).$ Conversely, if $X$ is not nearly pseudocompact,
it follows from the proof of the previous theorem that there
exists $N$, a copy of $\mathbb{N}$, $C$-embedded in $X$, contained
in a realcompact cozero set $P.$ There exists a continuous function $f$ satisfying $f(n)=n$, for all $n\in N$.
Choose $g\in C(X)$ such that $g=1$ on $N$ and equals to $0$ on a neighbourhood of 
$X\bs P.$ Then $cl_X (X\bs Z(fg)) \subseteq P$ is
hard in $X$. So $fg \in C_H (X)$ but $fg$ is unbounded on $N$. Hence $fg\notin C^*(X)$. 
\end{proof}
We begin our discussion with the following definitions of $hz$-filter
and $hz$-ultrafilter.
\begin{defn}
\label{def17}A $z$-filter is said to be $hz$-filter if it has a
base consisting of hard zero sets in $X$. Maximal $hz$-filter is
called $hz$-ultrafilter.
\end{defn}

Likewise we can define prime $hz$-filter. However it is not that
much relevant in this paper. It is further to be noted that the maximal
$hz$-filter and maximal $z$-filter containing an $hz$-filter are
essentially same as follows from the following proposition.
\begin{prop}
\label{prop18}Any $z$-filter containing a hard zero set is $hz$-filter.
\end{prop}

\begin{proof}
If $\script{B}$ is a base for a $z$-filter $\script{A}$ and $H$
be a hard zero set in $\script{A}$, then $\script{B}\cap H=\{H\cap B|B\in\script{B}\}$
is a base of $\script{A}$ consisting of hard zero sets. So this fact
immediately tells us that maximal $hz$-filter is the maximal $z$-filter
containing $hz$-filter.
\end{proof}
We have the following characterization of nearly pseudocompact spaces
using the ideas of $hz$-filter and $hz$-ultrafilters.
\begin{thm}
\label{thm19}The followings are equivalent.
\end{thm}

\begin{enumerate}
\item X is nearly pseudocompact
\item Every $hz$-filter is fixed
\item Every $hz$-ultrafilter is fixed.
\end{enumerate}
\begin{proof}
$(1)\imply(2)$ As $X$ is nearly pseudocompact, every hard set is
compact by theorem \ref{thm15}, every $hz$-filter is fixed.

$(2)\imply(3)$ Trivial.

$(3)\imply(1)$ Suppose $X$ is not nearly pseudocompact. By theorem
\ref{th16}, there exists $f\in C_{H}(X)$ which is unbounded. Then
$Z_{n}=\{x\in X:|f(x)|\geq n\}$ being closed subset of the hard set
$cl_{X}(X\bs Z(f))$ is hard zero sets in $X$ and hence can be extended
to a $hz$-ultrafilter which is not fixed. So we arrive at a contradiction.
\end{proof}
Let $HX$ be the family of all $hz$-ultrafilters on $X$ and $\eta X$
be the family of all non-$hz$-ultrafilters on $X$. Then $\beta X=HX\cup\eta X.$
We may recall here that $\beta X$ is the family of all $z$-ultrafilters
on $X$.

The following theorem is very important to us.
\begin{thm}
\label{thm20}Every real $hz$-ultrafilter is fixed. That is, $RX\subset\eta X.$
\end{thm}

\begin{proof}
Let $\script{U}$ be a real $hz$-ultrafilter. Let $H\in\script{U}$
be a hard zero set in $X$. Then $H\cap\script{U}=\{H\cap Z:Z\in\script{U}\}$
is a prime $z$-filter on $H$. Moreover as $\script{U}$ has countable
intersection property, $H\cap\script{U}$ has countable intersection
property. So $\bigcap(H\cap\script{U})$ is non empty as $H$ is realcompact
and hence every prime $z$-filter in $H$ with countable intersection
property is fixed. So $\script{U}$ is fixed.
\end{proof}
\begin{thm}
\label{thm21}If $q$ is a free $hz$-ultrafilter on $X.$ Then $q$
contains a regular hard zero set in $X.$
\end{thm}

\begin{proof}
Let $q$ be a free $hz$-ultrafilter on $X.$ Let $H$ be a hard zero
set in $q.$ Then there exists a compact set $K$ which satisfies
the property of hardness. Now every member of $q$ does not meet with
$K.$ If it would be so, $q$ turns out to be fixed $z$-ultrafilter.
So there exists a zero set $Z^{'}\in q$ such that $Z'\cap K=\emptyset.$
$X\bs Z'$ is therefore an open neighbourhood of $K$. Hence $Z'\cap H=H\backslash(X\bs Z')$
is completely separated by complement of a realcompact subset. So
there exists a regular zero set $Z_{0}$ containing $Z'\cap H$ and
is contained in a realcompact cozero set in $X.$ Then $Z_{0}$ is
a regular hard zero set containing $Z'\cap H.$ Hence $Z_{0}$ is
a member of $q$ as $Z'\cap H$ is a member of $q$.

We shall now develop a construction of nearly pseudocompactification
of a space $X$ in a formal way. For each space $X$, let $\delta X$
be the collection of all $hz$-ultrafilters along with every fixed
$z$-ultrafilter on $X$. So $\delta X=X\cup HX$. It is clear that
$\delta X$ is a subset of $\beta X$. Likewise we define $\overline{Z}=\{p\in\delta X:Z\in p\}$
for each $Z\in Z(X)$. Then $\{\overline{Z}:Z\in Z(X)\}$ is a base
for closed sets of some topology on $\delta X$. With respect to this
topology $\overline{Z}$ turns out to be $cl_{\delta X}Z$. In fact
this topology is indeed the subspace topology under $\beta X,$ as
already mentioned at the begining of this section. However treating
$\delta X$ as only a subspace of $\beta X$ does no way impede to
the progress of our motivation. We have later shown avoiding any properties
of $\beta X,$ that $\delta X$ is indeed a nearly pseudocompact space
in which $X$ is densely embedded.
\end{proof}
We already know that the following theorem is true for $\beta X$.
Without any modification it can be easily shown that the following
theorem is true in $\delta X$ also.
\begin{thm}
\label{thm22}Any two disjoint zero sets in $X$ has disjoint closure
in $\delta X$ and for any two zero sets $Z_{1},Z_{2}$ in $X$, $cl_{\delta X}Z_{1}\cap cl_{\delta X}Z_{2}=cl_{\delta X}(Z_{1}\cap Z_{2})$.
\end{thm}

\begin{proof}
Let $Z_{1}$ and $Z_{2}$ be two disjoint zero sets in $X$. Suppose
$p\in cl_{\delta X}Z_{1}\cap cl_{\delta X}Z_{2}$. Then $Z_{1}\in p$
and $Z_{2}\in p$, but $Z_{1}\cap Z_{2}=\phi$, a contradiction. Hence
$cl_{\delta X}Z_{1}\cap cl_{\delta X}Z_{2}=\emptyset$.

Let $p\in cl_{\delta X}Z_{1}\cap cl_{\delta X}Z_{2}$. So $Z_{1}$,
$Z_{2}\in p$. Hence $Z_{1}\cap Z_{2}\in p$. Thus $p\in\overline{Z_{1}\cap Z_{2}}$.
Other side follows trivially.
\end{proof}
\begin{thm}
\label{thm23}If $Z$ is a hard zero set, $cl_{\delta X}Z$ is compact
subset of $\delta X$ and also if $H$ be a hard set in $X$ then
$cl_{\delta X}H$ is compact.
\end{thm}

\begin{proof}
Let $Z$ be a hard zero set in $X$. Let $\{Z_{\alpha}:\alpha\in\lambda\}$
be a family of zero sets in $X$ such that $\{\overline{Z}_{\alpha}\cap\overline{Z}:\alpha\in\lambda\}$
is a family of basic closed sets in $\overline{Z}$ having finite
intersection property. Hence $\{Z_{\alpha}\cap Z:\alpha\in\lambda\}$
is a family of hard zero sets in $X$ satisfying finite intersection
property and hence can be extended to a $hz$-ultrafilter $q\in\delta X$.
As $Z_{\alpha},Z\in q$ for all $\alpha$, $q\in\overline{Z}$, $\overline{Z}_{\alpha}$
for all $\alpha$. So $q\in\bigcap_{_{\alpha}}Z_{\alpha}\cap Z$.
Hence $\overline{Z}$ is compact.

Second part: let $H$ be a hard set in $X$. There exists a compact
set $K$ in $H$ satisfying the property of hardness. Let $\{V_{\alpha}:\alpha\in\Lambda\}$
be an open cover of $cl_{\delta X}H$ in $\delta X$. There exists
$\alpha_{1},\alpha_{2},...,\alpha_{n}$ such that $K\subset\bigcup_{i=1}^{n}V_{\alpha_{i}}=W$.
Then there exists a zero set $Z$ and a realcompact cozero set $P$
such that $H\bs W\subset Z\subset P$. It is clear that $Z$ is a
hard zero set in $X$. As $W$ is open in $\delta X$, $cl_{\delta X}H\bs W\subset cl_{\delta X}Z$.
By the first part of the theorem, $cl_{\delta X}Z$ is compact. Hence
$cl_{\delta X}H\bs W$ is also compact. Another set of finitely many
members from the above cover are required to cover $cl_{\delta X}H\bs W$.
So $cl_{\delta X}H$ can be covered by finitely many members of $\{V_{\alpha}:\alpha\in\Lambda\}$.
Therefore $cl_{\delta X}H$ is compact.
\end{proof}
\begin{thm}
\label{thm24}Let $K$ be a compact set in $\delta X$, then $K\cap X$
is hard in $X$. If $H$ be a closed set in $X$ such that $cl_{\delta X}H$
is compact in $\delta X$, then $H$ is hard in $X.$
\end{thm}

\begin{proof}
Let $H=K\cap X$ is a closed subset of $X$ and let $T=K\cap cl_{X}(\eta X\cap X)$
is a compact subset of $H$. Let $V$ be an open neighbourhood of
$T$ in $\beta X$. Then $K\bs V$ is a compact set not intersecting
$cl_{\beta X}\eta X$. So there exist zero set neighbourhoods $Z_{1}$
and $Z_{2}$ of $K\bs V$ and $cl_{\beta X}\eta X$ respectively in
$\beta X$ with $Z_{1}\cap Z_{2}=\phi$. Then $cl_{\beta X}(Z_{2}\cap X)\subset\eta X$.
Then complement of $Z_{2}\cap X$ in $X$ is realcompact containing
$Z_{1}\cap X$, which again contains $H\bs(V\cap X)$. Hence $H$
is hard as any neighbourhood of $T$ in $X$ is indeed of the form
$V\cap X$ where $V$ is open in $\beta X$.

Second part is trivial as $cl_{\delta X}H\cap X=H$.
\end{proof}
\begin{thm}
A zero set is hard in $X$ if and only if it is not contained in any
z-filter which is not $hz$.
\end{thm}

\begin{proof}
If $H$ is hard zero set in $X$, as discussed above, any $z$-filter
containing $H$ is $hz$-filter. So it can not be member of any non-$hz$-filter.

For the converse, suppose $H$ is a zero set not containing in any
non-$hz$-filter on $X$. Let $\{\overline{Z}_{\alpha}\cap cl_{\delta X}H\}_{\alpha}$
be a family of basic closed sets in $cl_{\delta X}H$ having finite
intersection property. Then $\{Z_{\alpha}\}_{\alpha}\cup\{H\}$ satisfies
finite intersection property and hence can be extended to a $z$-ultrafilter
$q$. For all $\alpha$, $Z_{\alpha}\in q$ and $H\in q$. By assumption
$q$ is an $hz$-ultrafilter. So $q\in\delta X$. Hence $q\in\bigcap_{\alpha}(cl_{\delta X}H\cap\overline{Z}_{\alpha})$.
Hence $cl_{\delta X}H$ is compact. By theorem \ref{thm24} $H$ is
hard in $X$.
\end{proof}
\begin{thm}
$f\in C_{RC}(X)$ if and only if $Z(f)\in q$, for all $q\in RX$,
that is $RX\subseteq cl_{\beta X}Z(f)$.
\begin{proof}
Let $f\in C_{RC}(X)$. Suppose $Z(f)\notin p$ for some $p\in RX$.
There exists $Z\in p$ so that $Z\cap Z(f)=\emptyset$. Then $Z\subseteq X\bs Z(f)$
which is realcompact. So $Z$ is hard in $X$. Then $p\notin RX$,
a contradiction as every real $hz$-ultrafilter is fixed.

Conversely, suppose $f\in C(X)$ such that $Z(f)\in q$, for all $q\in RX$.
Let $\mathcal{U}$ be a $z$-ultrafilter in $X\backslash Z(f)=D$
(say) with countable intersection property. Then $i^{\#}(\mathcal{U})$
is a prime $z$-filter with countable intersection property and hence
is contained in a real $z$-ultrafilter (say) $p$ on $X.$ If $Z(f)\in p$,
then $\{x\in X:|f(x)|\leq\frac{1}{n}\}$ is a member of $p$, for
all $n$ and $\{x\in X:|f(x)|\geq\frac{1}{n}\}\notin p$, for all
$n.$ As $i^{\#}(\mathcal{U})$ is a prime $z$-filter, $\{x\in X:|f(x)|\leq\frac{1}{n}\}\in i^{\#}(\mathcal{U})$,
for all $n$. Hence $\{x\in X:|f(x)|\leq\frac{1}{n}\}\cap D\in\mathcal{U}$.
As $\mathcal{U}$ is real $z$-ultrafilter on $D$, $\bigcap_{n}\{x\in X:|f(x)|\leq\frac{1}{n}\}\cap D\neq\emptyset$.
So there exists a point $z\in D$, such that $|f(z)|\leq\frac{1}{n}$,
for all $n$. Hence $f(z)=0$, which is a contradiction. So $Z(f)\notin p$.
So $p$ can not be member of $RX.$ So $p$ must be fixed. As $Z(f)\notin p$,
by our convention ($*$), $p\notin Z(f)$. So $p\in D$. As $i^{\#}(\mathcal{U})\subset p$,
$p\in\cap i^{\#}(\mathcal{U})$. As $p\in D$ and $D$ being $z$-embedded,
$p\in\cap(i^{\#}(\mathcal{U})\cap D)=\mathcal{\cap U}$. So $\mathcal{\cap U}$
is fixed. Hence $D$ is realcompact and therefore $f\in C_{RC}(X)$.
\end{proof}
\end{thm}

\begin{thm}
If $p\in\delta X\backslash X$, then $p$ has a compact neighbourhood
in $\delta X.$
\end{thm}

\begin{proof}
: Let $p\in\delta X\backslash X.$ There exists a hard zero set $H$
in $p.$ Let $K$ be the compact subset of $H$ satisfying the property
of hardness. Let $U$ and $V$ be two $\delta X$- open sets such
that $p\in U$ and $K\subseteq V$ with $U\cap V=\emptyset$, so $p\notin cl_{\delta X}V$.
Now $H\backslash V=H\backslash(V\cap X)$ is contained in a zero set
$Z_{1}$ and is contained in a realcompact cozero set $X\backslash Z_{2}$,
where $Z_{2}$ is a zero set in $X.$ Then $p\notin cl_{\delta X}Z_{2}$
as $p\in cl_{\delta X}(H\backslash V)\subseteq Z_{1}$ and $Z_{1}\cap Z_{2}=\emptyset$.
Then there exists a zero set neighbourhood $W$ of $p$ in $\delta X$
such that $W\cap Z_{2}=\emptyset$. Then $W\cap X\subseteq X\backslash Z_{2}$.
Hence $W\cap X$ is hard in $X$ and $p\in W\subseteq cl_{\delta X}(W\cap X)$
which is compact. So $p$ has a compact neighbourhood in $\delta X.$
\end{proof}
\begin{thm}
$cl_{X}(\eta X\cap X)$ is closed in $\delta X$and is precisely the
set of points having no hard neighbourhood or equivalently any realcompact
neighbourhood.
\begin{proof}
It is clear that $cl_{X}(\eta X\cap X)$ is a closed in $X$. It is
enough to show that no point of $\delta X\bs X$ is a limit point
of $cl_{X}(\eta X\cap X)$. Let $p\in\delta X\bs X$. By previous
theorem there exists a compact neighbourhood $K$ of $p$ in $\delta X$
such that $p\in int_{\delta X}K\subseteq K$. Then $int_{\delta X}K\cap(\eta X\cap X)=\phi$,
otherwise suppose $x\in$ $int_{\delta X}K\cap(\eta X\cap X)\neq\phi$,
there exists a zero set $Z$ in $X$ containing $x$ and contained
in $int_{\delta X}K\cap(\eta X\cap X)$. Then $cl_{\delta X}Z$ being
a closed subset of $K$, is compact. So by theorem \ref{thm23} $Z$
is hard in $X$. As $x\in Z$, by our convention ($*$), $Z\in x$.
So $x$ is a fixed $hz$-ultrafilter and hence does not belongs to
$\eta X$, a contradiction. Therefore $int_{\delta X}K\cap(\eta X\cap X)=\phi$
which further implies $int_{\delta X}K\cap cl_{X}(\eta X\cap X)=\phi$.
Thus $cl_{X}(\eta X\cap X)$ is closed in $\delta X$.

Second part: it is clear that the set of points having hard neighbourhood
is an open subset of $X$ and does not intersect $(\eta X\cap X)$.
Hence it does not intersect $cl_{X}(\eta X\cap X)$. So no point of
$cl_{X}(\eta X\cap X)$ has hard (equivalently realcompact) neighbourhood
in $X$. Conversely suppose a point $p$ in $X$ does not belong to
$cl_{X}(\eta X\cap X)$, then $p$ does not belong to $cl_{\beta X}\eta X$.
Let $W$ and $L$ be zero set neighbourhoods of $cl_{\beta X}\eta X$
and $p$ respectively such that $W\cap L=\phi$ as $\beta X$ is Tychonoff.
Let $W\cap X=Z(f)$ and $L\cap X=Z(g)$. Then $\eta X$ is contained
in $cl_{\beta X}Z(f)$ and $Z(g)$ is a neighbourhood of $p$ in $X$.
Then $f\in C_{RC}(X)$. Since $Z(g)$ is contained in $X\bs Z(f)$
which is realcompact, $Z(g)$ is a hard neighbourhood of $p$. So
$cl_{X}(\eta X\cap X)$ is precisely the set of points having no hard
neighbourhood or equivalently any realcompact neighbourhood.
\end{proof}
\end{thm}

\begin{thm}
Union of hard and compact set is hard.
\end{thm}

\begin{proof}
As $H$ is hard, we have a compact set $T$ satisfying the property
of hardness. Let $K$ be a compact set. Then $W=K\cup T$ is also
compact. Then it is straight forward to check that any open neighbourhood
$G$ of $W$. $(H\cup K)\bs G$ can be completely separated from the
complement of realcompact cozero set. Hence $H\cup K$ is hard in
$X$.
\end{proof}
\begin{thm}
For any $p\in\eta X$, any $\beta X$- open neighbourhood of $p$
contains a free real $z$-ultrafilter.
\end{thm}

\begin{proof}
Let $p\in\eta X.$ So $p$ is a non-$hz$-ultrafilter on $X.$ For
$Z\in Z[X]$, let $p\in\beta X\backslash\overline{Z}$, Then $p$
is not a member of $\overline{Z}$, so $Z\notin p$. There exists
a member $W$ of $p$ such that $W\cap Z=\emptyset$. Suppose $Z$
belongs to all free real $z$-ultrafilter. Then $X\backslash Z$ is
realcompact. So $W$ is hard in $X$. This contradicts $p$ to be
non-$hz$-ultrafilter. So there exists a free real $z$-ultrafilter
$q$ such that $q\notin\overline{Z}$. Hence $q\in\beta X\backslash\overline{Z}$.
\end{proof}
The above theorem shows that $cl_{\beta X}\eta X$ is precisely the
$cl_{\beta X}RX.$

\begin{thm}
$\delta X$ is nearly pseudocompact.
\end{thm}

\begin{proof}
Let $\delta X_{lc}$ be the set of points in $\delta X$ having compact
neighbourhood. Then $\delta X\backslash X\subseteq\delta X_{lc}$.
As $\delta X_{lc}$ is an open subset of $\delta X,$ $\delta X_{lc}\cap X\neq\emptyset.$
And $cl_{\delta X}(\delta X_{lc}\cap X)\supset\delta X_{lc}.$ Hence
$cl_{\delta X}\delta X_{lc}=cl_{\delta X}(\delta X_{lc}\cap X).$
Let $\delta X_{1}=cl_{\delta X}\delta X_{lc}.$ Let $f\in C(\delta X)$
such that $f$ is unbounded on $\delta X_{1}.$ Then $f$ is unbounded
on $\delta X_{lc}\cap X$. Hence there exists a copy $N$ of $\mathbb{N}$,
$C$-embedded in $\delta X.$ Then each point of $N$ is contained
in realcompact cozero set. As countable union of realcompact cozero
sets (infact $z$-embedded sets) is realcompact cozero sets, $N$
is contained in a realcompact cozero set in $\delta X.$ As $N$ is
$C$-embedded in $\delta X,$ $N$ can be completely separated from
the complement of the realcompact cozero set. Hence $N$ is hard in
$\delta X$. Hence $N$ is compact which is absurd. So $\delta X_{lc}$
is relatively pseudocompact and hence $\delta X_{1}$ is pseudocompact.
Let $\delta X_{2}=\delta X\backslash\delta X_{lc}$is a regular closed
subset of $\delta X$ which is anti-locally realcompact. Clearly $int_{\delta X}(\delta X_{1}\cap\delta X_{2})=\emptyset.$
So $\delta X$ is nearly pseudocompact.
\end{proof}

\section{Uniqueness of $\delta X$ with respect to extension property}

In this section we shall show that the nearly pseudocompactification,
$\delta X$, of $X$ is unique, with respect to some properties. Prior
to that, we hereby introduce a subring $S(X)$ of $C(X)$ containing
$C^{*}(X)$.
\begin{defn}
Let $S(X)=\{f\in C(X):f$ is bounded on every hard set in $X\}$.
We call such $f$ as hard-bounded continuous map.
\end{defn}

Then clearly $C^{*}(X)\subset S(X)\subset C(X)$. Moreover as closed
subset of hard set is hard in $X$ and finite union of hard sets is
again a hard set, it follows that $S(X)$ is indeed a subring of $C(X)$.

We call a space $X$ to be $S$-embedded in $Y$ if every hard-bounded
continuous map on $X$ can be continuously extended to $Y$.
\begin{defn}
A continuous map $f:X\rightarrow Y$ is called $h2pc$ map if $f$
takes a hard set to a pre-compact set in $Y$. i.e. if $H$ is hard
in $X$ then $cl_{Y}f(H)$ is compact.
\end{defn}

\begin{example}
(i) Any hard map \cite{r76}(that takes hard set to a hard set ) from
a space to a nearly pseudocompact space is $h2pc$.

(ii) Every hard-bounded continuous map is $h2pc$: Infact, if $H$
is hard in $X$, $f$ is a hard-bounded continuous map, then $f(H)$
is bounded subset of $\mathds{R\Rightarrow}cl_{\mathds{R}}f(H)$ is
compact.

(iii) Then the inclusion map $i:X\rightarrow\delta X$ is also $h2pc$
map.

(iv) Any continuous map from a nearly pseudocompact space is also
$h2pc$ map: Infact if $H$ is hard in nearly pseudocompact space
$X$, $H$ is compact. Let $f:X\rightarrow Y$ be a continuous map.
Then $f(H)$ is compact $\Rightarrow cl_{Y}f(H)=f(H)$ is also compact.
So $f$ is $h2pc$.
\end{example}

\begin{thm}
Composition of two $h2pc$ map is also $h2pc$.
\end{thm}

\begin{proof}
Let $f:X\rightarrow Y$ and $g:Y\rightarrow Z$ be $h2pc$ maps. Let
$H$ be a hard set in $X$. Then $cl_{Y}f(H)$ is compact in $Y$.
By continuity, $g(cl_{Y}f(H))\subseteq cl_{Z}g(f(H))$. Now we have
$g(f(H))\subseteq g(cl_{Y}f(H))\subseteq cl_{Z}g(f(H))$. But $g(cl_{Y}f(H))$
is compact. Hence $g(cl_{Y}f(H))=cl_{Z}g(f(H))\Rightarrow gf$ is
$h2pc$ map.
\end{proof}
\begin{thm}
\label{thm35}Let $f:X\to Y$be a hard map and $g:Y\to Z$ be an $h2pc$
map. Then $g\circ f:X\to Z$ is an $h2pc$ map.
\end{thm}

Though the proof of the above theorem is trivial, its importance is
reflected in the last paragaraph of this section.
\begin{thm}
Let $X$ be dense in $T$ and each point of $T\bs X$ is in $T$-closure
of a hard set in $X$. Then each point of $T\bs X$ is in the $T$-closure
of hard zero set (equivalently regular hard zero set) in $X$.
\end{thm}

\begin{proof}
Let $p\in T\bs X$ and $H$ be a hard set in $X$ such that $p\in cl_{T}H$.
Let $K$ be the compact subset of $H$ satisfying the property of
hardness. As $p\notin K$, $\exists$ open sets $U$, $V$ in $T$
such that $p\in U$ and $K\subset V$ with $U\cap V=\phi$. Since
$K$ satisfies the property of hardness, there exist a zero set $Z_{1}$
and a cozero set $X\bs Z_{2}$ in $X$ where $Z_{2}$, a zero set,
such that $H\bs V\subset Z_{1}\subset X\bs Z_{2}$ and $X\bs Z_{2}$
is realcompact. This implies that $Z_{1}$ is a hard zero set in $X$.
(Infact, here we may even choose $Z_{1}$ to be regular zero set.
Then $Z_{1}$ becomes a regular hard zero set). Now $p\notin cl_{T}(V\cap X)$
as $U\cap V=\phi$, but $p\in cl_{T}H.$ So $p\in cl_{T}(H\bs V)$.
Hence $p\in cl_{T}Z_{1}$.
\end{proof}
\begin{thm}
Suppose $X$ is dense in $T$ such that every point $p\in T\bs X$
is in the $T$-closure of some hard set in $X$. Then the followings
are equivalent.
\end{thm}

\begin{enumerate}
\item If $\tau$ is a $h2pc$ map from a space $X$ to $Y$ then $\tau$
can be continuously extended upto $T$.
\item $X$ is $S$-embedded in $T$.
\item For any two zero sets $Z_{1}$, $Z_{2}$ in $X$, $Z_{1}\cap Z_{2}=\phi\Rightarrow cl_{T}Z_{1}\cap cl_{T}Z_{2}=\phi$.
\item For any two zero sets $Z_{1}$, $Z_{2}$ in $X$, $cl_{T}(Z_{1}\cap Z_{2})=cl_{T}Z_{1}\cap cl_{T}Z_{2}$.
\item Each point of $T$ is the limit of either a fixed $z$-ultrafilter
or a $hz$-ultrafilter on $X$.
\end{enumerate}
\begin{proof}
$(1)\Rightarrow(2):$ Trivial as every hard-bounded continuous map
is $h2pc$ map.

$(2)\Rightarrow(3):$ Trivially follows from Theorem 6.4 \cite{hr80}
as $C^{*}(X)\subset S(X)$.

$(3)\Rightarrow(4):$ Trivially follows from Theorem 6.4 \cite{hr80}.

$(4)\Rightarrow(5):$ Let $p\in T$. If $p\in X$, then $p$ is the
limit of unique fixed $z$-ultrafilter $A_{p}=\{Z\in Z[X]:p\in Z\}$.
So suppose now, $p\in T\bs X$. There exists a hard zero set $H$
in $X$ such that $p\in cl_{T}H$. Let $\mathcal{N}_{p}$ be the family
of all zero set neighbourhoods of $p$. Then $\{Z\cap X:Z\in\mathcal{N}_{p}\}\cup\{H\}$
satisfies finite intersection property. Hence can be extended to a
$z$-ultrafilter $\mathcal{U}$ on $X$. As $H\in\mathcal{U}$, $\mathcal{U}$
is then a $hz$-ultrafilter. That $\mathcal{U}$ is unique and converges
to $p$, follows the same line of arguments as given in Theorem 6.4
\cite{hr80}.

$(5)\Rightarrow(1):$ Let $\tau:X\rightarrow Y$ be a $h2pc$ map.
Let $p\in T$. Then $\mathcal{U}\rightarrow p$ where $\mathcal{U}$
is either a fixed $z$-ultrafilter on $X$ or $\mathcal{U}$ is a
$hz$-ultrafilter on $X$ by $(5)$. Now $\tau^{\#}(\mathcal{U})$
is a prime $z$-filter on $Y$. If $\tau^{\#}(\mathcal{U})$ is fixed
then there exists $y_{p}\in Y$ such that $\bigcap\tau^{\#}(\mathcal{U})=\{y_{p}\}$.
Then we shall define $\tau^{*}(p)=y_{p}$. Now if $\mathcal{U}$ is
fixed then $p\in X$ and $\mathcal{U}=A_{p}$. $Z\in\tau^{\#}(\mathcal{U})\Leftrightarrow\tau^{-1}(Z)\in A_{p}\Leftrightarrow p\in\tau^{-1}(Z)\Leftrightarrow\tau(p)\in Z.$
So $\tau(p)\in\bigcap\tau^{\#}(\mathcal{U})\Rightarrow\tau^{*}(p)=\tau(p)$.
This also shows that $\tau^{*}$ is an extension of $\tau$. Now suppose
$p\in T\bs X$, then there exists unique $hz$-ultrafilter $\mathcal{U}$
on $X$ which converges to $p$. Let $Z\in\tau^{\#}(\mathcal{U})\Rightarrow\tau^{-1}(Z)\in\mathcal{U}\Rightarrow$there
exists a hard zero set $H_{Z}$ in $\mathcal{U}$ such that $H_{Z}\subset\tau^{-1}(Z)\Rightarrow\tau(H_{Z})\subseteq Z\Rightarrow cl_{Y}(\tau(H_{Z}))\subseteq Z$.
Now $cl_{Y}\tau(H_{Z})$ is compact and $\{cl_{Y}\tau(H_{Z}):Z\in\tau^{\#}(\mathcal{U})\}$
satisfies finite intersection property. Hence $\bigcap_{Z\in\mathcal{U}}cl_{Y}\tau(H_{Z})\neq\phi$.
Clearly $\bigcap_{Z\in\mathcal{U}}cl_{Y}\tau(H_{Z})\subseteq\bigcap\tau^{\#}(\mathcal{U})$.
So $\bigcap\tau^{\#}(\mathcal{U})\neq\phi$. Let $y_{p}\in\tau^{\#}(\mathcal{U})\Rightarrow\tau^{\#}(\mathcal{U})=\{y_{p}\}.$
So we have defined $\tau^{*}:\delta X\rightarrow Y$ which is an extension
of $\tau$. That $\tau^{*}$ is continuous, follows again the same
line of arguments given in Theorem 6.4 \cite{hr80}.
\end{proof}
By construction of $\delta X$, it is clear that $\delta X$ is precisely
the collection of limits of all fixed $z$-ultrafilter and $hz$-ultrafilter
on $X$. Further if $p\in\delta X$, there exists a compact neighbourhood
$K$ of $p$ in $\delta X$. Then $K\cap X$ is hard in $X$ and $p\in cl_{\delta X}(K\cap X)$.
So from the previous theorem we conclude the following theorem.
\begin{thm}
\label{thm39}For each space $X$, there exists a nearly pseudocompactification
$\delta X$ of $X$ which satisfies the following equivalent properties.
\end{thm}

\begin{enumerate}
\item Any $h2pc$ map from a space $X$ to $Y$ has continuous extension
upto $\delta X$.
\item $X$ is $S$-embedded in $\delta X$.
\item For any two zero sets $Z_{1},$$Z_{2}$ in $X$, $Z_{1}\cap Z_{2}=\phi\Rightarrow cl_{\delta X}Z_{1}\cap cl_{\delta X}Z_{2}=\phi$.
\item For any two zero sets $Z_{1},$$Z_{2}$ in $X$, $cl_{\delta X}(Z_{1}\cap Z_{2})=cl_{\delta X}Z_{1}\cap cl_{\delta X}Z_{2}$.
\item Every fixed $z$-ultrafilter and $hz$-ultrafilter converges in $\delta X$.
\end{enumerate}
Next theorem shows that $\delta X$ is unique with respect to certain
properties.
\begin{cor}
Let $X$ be dense in $Y$ such that any $h2pc$ map on $X$ has continuous
extension to $Y$ and $Y$-closure of each hard set in $X$ is compact
in $Y$, then $Y$ is homeomorphic with $\delta X$.
\end{cor}

\begin{proof}
We note that the inclusion map $i_{\delta X}:X\rightarrow\delta X$
is $h2pc$ map and $\delta X$-closure of hard set in $X$ is compact.
As per the conditions, the inclusion map $i_{Y}:X\rightarrow Y$ is
also a $h2pc$ map. Now $\exists$ $f:\delta X\rightarrow Y$ and
$g:Y\rightarrow\delta X$ such that $f|_{X}=I_{X}$ and $g|_{X}=I_{X}$.
Then $f\circ g=I_{Y}$ and $g\circ f=I_{\delta X}\Rightarrow f$ is
a homeomorphism. So $Y$ is homeomorphic with $\delta X$.
\end{proof}
\begin{cor}
If $f:X\to Y$ be a hard map from a space $X$ to a nearly pseudocompact
space $Y,$then $f$ can be uniquely extended to a continuous map
$f^{*}:\delta X\to Y$.
\end{cor}

\begin{proof}
It trivially follows from the theorem \ref{thm39}, as any hard map
from a space to a nearly pseudocompact space is $h2pc$ map.
\end{proof}
Acknowledgements: The authors sincerely acknowledge the support received
from DST FIST programme (File No. SR/FST/MSII/2017/10(C))


\begin{thebibliography}{1}
\bibitem{r76} M.C.Rayburn - On hard sets, General Topology and its
Applications 6(1976),21-26.

\bibitem{gj60} L. Gillman and M. Jerison, Rings of Continuous Functions,
University Series in Higher Math, Van Nostrand, Princeton, New Jersey,1960.

\bibitem{hr80} M. Henriksen and M. Rayburn - On nearly pseudocompact
spaces,Top.Appl. 11 (1980),161-172.

\bibitem{hr87} Melvin Henriksen and Marlon C.Rayburn - Nearly pseudocompact
extensions, Math Japonica 32,No.4(1987), 569-582.

\bibitem{ma05} B. Mitra and S.K. Acharyya, Characterizations of Nearly
Pseudocompact spaces and Related spaces (With S.K. Acharyya), Topology
Proceedings, Vol 29, No. 2, 2005, 577 - 594

\bibitem{mc20} B. Mitra and D. Chowdhury, A characterization of C-type
subrings of C(X) of some kind (With Debojyoti Chowdhury), Positivity,
volume 24, (2020) 1181--119

\bibitem{m71} M. Mandelker - Supports of continuous functions,Trans.
Amer. Math. Soc 156 (1971), 73-83

\bibitem{mw75} R.C. Walker - The Stone-\$\textbackslash check\{C\}\$ech
Compactification, North-Holland Publishing Company, Amsterdam,1975
\end{thebibliography}
\end{document}